\documentclass{article}

\usepackage{times}
\usepackage{mathptmx}
\usepackage{amsmath,amssymb,amsthm} 
\usepackage{bbm}           

\newcommand{\mathscr}{\mathcal}
\renewcommand{\rho}{\varrho}

\DeclareMathOperator*{\argmin}{arg\,min}
\DeclareMathOperator{\prox}{prox}

\newcommand{\Rb}{\mathbbm{R}}

\newcommand{\widebar}{\bar}

\newtheorem{theorem}{Theorem}
\newtheorem{lemma}[theorem]{Lemma}
\newtheorem{assumption}[theorem]{Assumption}
\newtheorem{corollary}[theorem]{Corollary}

\begin{document}

\title{Rate of Convergence of the Bundle Method}


\author{
 Yu Du\footnote{Department of Management Science and Information Systems, 100 Rockefeller Road,  Rutgers University, Piscataway, NJ 08854, USA;
               Email: duyu@rutgers.edu}
\and
 Andrzej Ruszczy\'nski\footnote{Department of Management Science and Information Systems, 100 Rockefeller Road,  Rutgers University, Piscataway, NJ 08854, USA;
               Email: rusz@rutgers.edu}
}

\date{September 3, 2016}

\maketitle

\begin{abstract}
We prove that the bundle method for nonsmooth optimization achieves
solution accuracy $\varepsilon$ in at most $\mathcal{O}\big(\ln(1/\varepsilon)/\varepsilon\big)$ iterations, if the function is strongly convex. The result is true for the versions of the method with
multiple cuts and with cut aggregation.\\
\end{abstract}


\section{Introduction}

The objective of this note is to provide a worst-case bound on the rate of convergence of the bundle method for solving
convex optimization problems of the following form:
\begin{equation}
\label{problem}
\min_{x\in \Rb^n} F(x),
\end{equation}
where $F:\Rb^n\to \Rb$ is a convex function. The only additional assumption about the function needed to bound the rate is  strong convexity of the function about the minimum point.

The bundle methods were developed
in \cite{lem:iiasa,mif:82}. First rigorous convergence analysis and versions
with cut aggregation were provided in \cite{kiwiel1983aggregate,Kiwiel-book}.
For a comprehensive treatment of bundle and trust region methods, see \cite{bogilesa:on,HUL-book}.
Although the bundle method is a method of choice for nonsmooth optimization, no general rate of convergence results are available.
This is due to the complicated structure of the method, in which successive iterations carry out different operations,
depending on the outcome of a sufficient descent test.

Some results on the rate of convergence are available for the related bundle level method \cite{LNN95}, which achieves  $O(1/\varepsilon^2)$ iteration complexity for general nonsmooth convex programming problems.  Similar results have been obtained for modified versions in \cite{KKC95} and \cite{LG15}.

Our contribution is to prove at most $\mathcal{O}\big(\ln(1/\varepsilon)/\varepsilon\big)$ iteration complexity of the
classical bundle method, under the condition of strong convexity about the minimum point. This is achieved by bounding the
numbers of null steps between successive descent steps, and integrating these bounds across the entire run of the method. The result holds true for two versions of the method: with
multiple cuts and with cut aggregation.

In section \ref{s:method}, we present both versions of the bundle method and recall its convergence properties.
Section \ref{s:aux} contains several auxiliary results. A  worst-case bound
 on the convergence rate of the method is
 derived in section~\ref{s:rate}.

We use $\langle\cdot,\cdot\rangle$ and $\|\cdot\|$ to denote the usual scalar product and the Euclidean norm in a finite
dimensional space.

\section{The Bundle Method}
\label{s:method}

The bundle method is related to the fundamental idea of the \emph{proximal point method}, which uses
the \emph{Moreau--Yosida regularization} of $F(\cdot)$,
\begin{equation}
\label{MY-regularization}
F_\rho(y) = \min_x \Big\{ F(x) +  \frac{\rho}{2}\big\|x-y\big\|^2 \Big\},\quad \rho>0,
\end{equation}
to construct the \emph{proximal step} for \eqref{problem},
\begin{equation}
\label{prox}
\prox_F(y) = \argmin_x \Big\{ F(x) +  \frac{\rho}{2}\big\|x-y\big\|^2 \Big\}.
\end{equation}
The proximal point method carries out the
iteration $x^{k+1} = \prox_F(x^k)$, $k=1,2,\dots$ and is known to converge to a minimum of $F(\cdot)$, if a minimum exists \cite{rockafellar1976monotone}.

The main idea of the bundle method is to replace problem \eqref{problem} with a sequence of approximate problems
of the following form:
\begin{equation}
\label{subproblem}
\min_x \widetilde{F}^{k}(x) + \frac{\rho}{2}\big\|x-x^k\big\|^2.
\end{equation}
Here $k=1,2,\dots$ is the iteration number, $x^k$ is the current best approximation to the solution, and
$\widetilde{F}^{k}(\cdot)$ is a piecewise linear convex lower approximation of the function $F(\cdot)$. Two versions of the method
differ in the way this approximation is constructed.

\subsection{The Version with Multiple Cuts}

In the version with multiple cuts, the approximations $\widetilde{F}^k(\cdot)$ are constructed as follows:
\[
\widetilde{F}^k(x) = \max_{j\in J_k} \big\{ F(z^j) + \langle g^j, x - z^j\rangle \big\},
\]
with some previously generated points $z^j$ and subgradients
$g^j\in \partial F(z^j)$, $j\in J_k$, where $J_k\subseteq\{1,\dots,k\}$. The points $z^j$ are
solutions of problems \eqref{subproblem} at earlier iterations of the method.

Thus, problem \eqref{subproblem} differs from
\eqref{MY-regularization} by the fact that the  function $F(\cdot)$ is replaced by
a cutting plane approximation. The other difference between the bundle method and the proximal point method
is that the solution $z^{k+1}$ of problem \eqref{subproblem} is subject to a sufficient improvement test, which
decides whether the next proximal center $x^{k+1}$ should be set to $z^{k+1}$ or remain unchanged.\\

\noindent
\textbf{Bundle Method with Multiple Cuts}\vspace{1ex}\\
\textbf{Step 0:} Set $k=1$, $J_1=\{1\}$,
$z^1=x^1$, and select $g^1\in \partial F(z^1)$.  Choose parameter $\beta \in (0,1)$, and a stopping precision $\varepsilon > 0$. \vspace{1ex}\\
\textbf{Step 1:} Find the solution $z^{k+1}$ of subproblem \eqref{subproblem}.\\
\textbf{Step 2:}  If
\begin{equation}\label{eqn: n_v}
F(x^k) - \widetilde{F}^k(z^{k+1}) \le \varepsilon,
\end{equation}
then stop; otherwise, continue.\vspace{1ex}\\
\textbf{Step 3:} If
\begin{equation}\label{eqn: n_new_rule1}
F(z^{k+1}) \leq F(x^k) - \beta\big(F(x^k) - \widetilde{F}^k(z^{k+1})\big),
\end{equation}
then set $x^{k+1} = z^{k+1}$ (\emph{descent step}); otherwise set $x^{k+1}=x^k$ (\emph{null step}).\vspace{1ex}\\
\textbf{Step 4:} Select a set $J_{k+1}$ so that
\[
J_k \cup\{k+1\} \supseteq J_{k+1} \supseteq \{k+1\}\cup \big\{j\in J_k:
F(z^j) + \langle g^j, z^{k+1} - z^j\rangle = \widetilde{F}^k(z^{k+1})\big\}.
\]
Increase $k$ by 1 and go to Step 1.\vspace{1ex}

\subsection{The Version with Cut Aggregation}

In the version with cut aggregation,
as described in \cite{kiwiel1983aggregate} and \cite[sec. 7.4.4]{ruszczynski2006nonlinear},
the approximations $\widetilde{F}^k(\cdot)$ have only two pieces:
\[
\widetilde{F}^k(x) = \max\big\{ \widebar{F}^k(x), F(z^k) + \langle g^k, x - z^k\rangle \big\},
\]
with the last generated point $z^k$ and the corresponding subgradient
$g^k\in \partial F(z^k)$. The function $\widebar{F}^k(x)$ is a convex combination of affine minorants
 $F(z^j) + \langle g^j, x - z^j\rangle$, constructed at previously generated points $z^j$ with subgradients
$g^j\in \partial F(z^j)$, where $1\le j <k$. This function is updated at each iteration, as specified in Step 4 of the algorithm below.\\

\noindent
\textbf{Bundle Method with Cut Aggregation}\vspace{1ex}\\
\textbf{Step 0:} Set $k=1$, 
$z^1=x^1$, $\bar{F}^1(\cdot)\equiv\ -\infty$, and select $g^1\in \partial F(z^1)$.  Choose parameter $\beta \in (0,1)$, and a stopping precision $\varepsilon > 0$. \vspace{1ex}\\
\textbf{Step 1:} Find the solution $z^{k+1}$ of subproblem \eqref{subproblem}.\\
\textbf{Step 2:}  If
\begin{equation}\notag
F(x^k) - \widetilde{F}^k(z^{k+1}) \le \varepsilon,
\end{equation}
then stop; otherwise, continue.\vspace{1ex}\\
\textbf{Step 3:} If
\begin{equation}\notag
F(z^{k+1}) \leq F(x^k) - \beta\big(F(x^k) - \widetilde{F}^k(z^{k+1})\big),
\end{equation}
then set $x^{k+1} = z^{k+1}$ (\emph{descent step}); otherwise set $x^{k+1}=x^k$ (\emph{null step}).\vspace{1ex}\\
\textbf{Step 4:} Define
\begin{equation}
\label{Fk-update}
\widebar{F}^{k+1}(x) = \theta_k\widebar{F}^k(x) +(1-\theta_k)\big[ F(z^k) + \langle g^k, x - z^k\rangle\big],
\end{equation}
where $\theta_k\in [0,1]$ is such that the gradient of $\widebar{F}^{k+1}(\cdot)$ is equal to the subgradient of $\widetilde{F}^k(\cdot)$
at $z^{k+1}$ that satisfies the optimality conditions for problem \eqref{subproblem}.
Increase $k$ by 1 and go to Step 1.\vspace{1ex}


\subsection{Convergence}
Convergence of the bundle method (in both versions) for convex functions is well-known.
\begin{theorem}
\label{c:summary}
Suppose ${\rm Argmin}\,F\ne \emptyset$ and $\varepsilon=0$. Then a point $x^*\in {\rm Argmin}\,F$ exists, such that:
\[
\lim_{k\to\infty} x^k = \lim_{k\to\infty} z^k = x^*.
\]
\end{theorem}
\begin{proof} The proof of this result (in slightly different versions) can be found in numerous references, such as  \cite[Thm. 4.9]{Kiwiel-book},  \cite[Thm. XV.3.2.4]{HUL-book}, or \cite[Thm. 7.16]{ruszczynski2006nonlinear}.
\end{proof}

\section{Auxiliary results}
\label{s:aux}

In this section, we collect several auxiliary results on the properties of the bundle method in the general case. They are either refined versions or direct quotations of results presented
in \cite[sec. 7.4]{ruszczynski2006nonlinear}. We consider both versions of the method in parallel, with the corresponding versions of the functions
$\widetilde{F}^k(\cdot)$. All the results hold true for both versions,
because the analysis of the method with multiple cuts uses the version with cut aggregation anyway; in the proofs we explain the minor differences between the methods.

We first prove that if a null step occurs at iteration $k$, then the optimal objective function values of consecutive subproblems are increasing, and the gap is bounded below by a quantity dependent on
\begin{equation}
\label{vk-def}
v_k = F(x^k) - \widetilde{F}^k(z^{k+1}).
\end{equation}
We define the optimal objective function values of subproblem \eqref{subproblem} at iteration $k$ as:
\begin{equation}
\eta^k = \widetilde{F}^{k}(z^{k+1})+ \frac{\rho}{2}\big\|z^{k+1}-x^k\big\|^2.
\label{eta_k}
\end{equation}
Note that $x^{k+1} = x^k$ at a null step.

Since the point $z^{k+1}$ is the optimal solution of  \eqref{subproblem} at iteration $k$,
the vector
\begin{equation}
\label{skp1}
s^{k+1} = - \rho \big(z^{k+1} - x^k \big).
\end{equation}
is the subgradient  of $\widetilde{F}^k(\cdot)$ at $z^{k+1}$
that features in the optimality conditions.
%
Consequently,
 the point $z^{k+1}$ is also the unique minimum of
the problem
\begin{equation}
\label{eta-sk}
\min_{x} \Big\{ \widetilde{F}^k(z^{k+1}) + \langle s^{k+1},x-z^{k+1} \rangle + \frac{\rho}{2}\big\|x - x^k\big\|^2 \Big\},
\end{equation}
and the values of \eqref{eta_k} and \eqref{eta-sk} coincide. In the method with cut aggregation, by the
definition of $\theta_k$ in \eqref{Fk-update} and by \eqref{skp1}, we have
\[
\widebar{F}^{k+1}(x)=\widetilde{F}^k(z^{k+1}) + \langle s^{k+1},x-z^{k+1} \rangle.
\]
The addition of a new cut at
$z^{k+1}$ and possible deletion of inactive cuts (in the method without cut aggregation), creates a function $\widetilde{F}^{k+1}(\cdot)$, which
satisfies the inequality
\begin{equation}
\label{aggregate-lower}
\widetilde{F}^{k+1}(x) \ge \max\big(\widetilde{F}^k(z^{k+1}) + \langle s^{k+1},x-z^{k+1} \rangle, F(z^{k+1}) + \langle g^{k+1}, x - z^{k+1} \rangle \big).
\end{equation}
In the method with cut aggregation, exact equality in \eqref{aggregate-lower} is true, but we use the inequality ``$\ge$'' in further considerations.
Since the test for a descent step is not satisfied, we have
\[
\widetilde{F}^{k+1}(z^{k+1})=F(z^{k+1})> \widetilde{F}^{k}(z^{k+1}).
 \]
 The solution $z^{k+1}$ of problem \eqref{eta-sk} is unique, due to the strong convexity of the function being minimized there. Therefore, the optimal value of \eqref{eta-sk} must increase after replacing $\widetilde{F}^k(z^{k+1}) + \langle s^{k+1},x-z^{k+1} \rangle$
 with the right hand side of \eqref{aggregate-lower}. The optimal value $\eta^{k+1}$ of \eqref{subproblem}
 at iteration $k+1$ is at least as large, due to \eqref{aggregate-lower}.

The key issue is to bound the actual increment from $\eta^k$ to $\eta^{k+1}$ from below.
\begin{lemma}\label{lem: increment_bound}
If a null step is made at iteration $k$, then
\begin{equation}
\label{eta-increase}
\eta^{k+1} \ge \eta^k + \frac{1-\beta}{2}\bar{\mu}_kv_k,
\end{equation}
where
\begin{equation}
\label{muk-def}
\bar{\mu}_k =
\min\bigg\{1, \frac{(1-\beta)\rho v_k}{\|s^{k+1} - g^{k+1}\|^2}\bigg\}.
\end{equation}
\end{lemma}
\begin{proof}
Using \eqref{aggregate-lower}, we can bound the optimal value of the subproblem \eqref{subproblem}
 at iteration $k+1$ as follows:
\begin{equation}\label{eqn: f2_relax}
\begin{aligned}
\eta^{k+1}
&\ge \min_{x}\Big\{\max \Big(  \widetilde{F}^k(z^{k+1}) + \langle s^{k+1},x-z^{k+1} \rangle,\\
&\qquad F(z^{k+1}) + \langle g^{k+1}, x - z^{k+1}\rangle\Big) +   \frac{\rho}{2}\big\|x-x^k\big\|^2\Big\}\\
&\ge \min_{x} \Big\{
(1-\mu) \Big( \widetilde{F}^k(z^{k+1}) + \langle s^{k+1},x-z^{k+1} \rangle\Big) \\
&{\quad} + \mu\Big( F(z^{k+1}) + \langle g^{k+1}, x - z^{k+1} \rangle \Big) +  \frac{\rho}{2}\big\|x-x^k\big\|^2\Big\},
\end{aligned}
\end{equation}
with any value of the parameter $\mu\in [0,1]$.
 Define
 \begin{multline}
 \label{Qk}
 \hat{Q}_k(\mu) = \min_{x} \Big\{
(1-\mu) \Big( \widetilde{F}^k(z^{k+1}) + \langle s^{k+1},x-z^{k+1} \rangle\Big) \\
{\quad} + \mu\Big( F(z^{k+1}) + \langle g^{k+1}, x - z^{k+1} \rangle \Big) +  \frac{\rho}{2}\big\|x-x^k\big\|^2\Big\}.
 \end{multline}
  Due to \eqref{eta-sk},  $\hat{Q}_k(0)=\eta^k$. It follows from \eqref{eqn: f2_relax} that
the difference between $\eta^{k+1}$ and $\eta^k$ can be bounded from below by the increase in the optimal value $\hat{Q}_k(\mu)$, when $\mu$ moves away from zero. That is,
\[
\eta^{k+1} - \eta^k \geq \max_{\mu\in[0,1]}\hat{Q}_k(\mu) - \hat{Q}_k(0).
\]
By direct calculation and with a view to \eqref{skp1}, the minimizer on the right hand side of \eqref{Qk} is
\[
\hat{x}(\mu)= z^{k+1}+ \frac{\mu}{\rho} \big(s^{k+1}-g^{k+1}\big).
 \]
To obtain the derivative of $\hat{Q}_k(\cdot)$, we calculate the partial derivative of
the right-hand side of \eqref{Qk} with respect to $\mu$ and then substitute $x=\hat{x}(\mu)$. We obtain
\begin{align*}
 \hat{Q}'_k(\mu) &= F(z^{k+1}) - \widetilde{F}^k(z^{k+1})  + \langle g^{k+1} - s^{k+1}, \hat{x}(\mu)-z^{k+1} \rangle \\
 &=
  F(z^{k+1}) - \widetilde{F}^k(z^{k+1}) -\frac{\mu}{\rho} \big\|s^{k+1} - g^{k+1}\big\|^2 .
\end{align*}
Thus, for any value of $\mu_k\in [0,1]$,
\begin{align*}
\eta^{k+1} - \eta^k &\geq \hat{Q}_k({\mu}_k)-\hat{Q}_k(0) = \int^{{\mu}_k}_0 \hat{Q}'_k(\mu)\;d\mu \\
&=  {\mu}_k\left(F(z^{k+1}) - \widetilde{F}^k(z^{k+1}) - \frac{{\mu}_k}{2\rho}\big\|s^{k+1} - g^{k+1}\big\|^2\right).
\end{align*}
Define
\[
{\mu}_k = \min\bigg\{1, \frac{\rho\big(F(z^{k+1}) - \widetilde{F}^k(z^{k+1})\big)}{\|s^{k+1}- g^{k+1}\|^2}\bigg\}.
\]
 Clearly, $ {\mu}_k\in[0,1]$.
Substitution into the last displayed relation implies the inequality
\begin{equation}\label{eqn: gap}
\eta^{k+1} - \eta^k  \ge \frac{{\mu}_k}{2} \big( F(z^{k+1}) - \widetilde{F}^k(z^{k+1}) \big).
\end{equation}
If a null step occurs at iteration $k$, then the update step rule (\ref{eqn: n_new_rule1}) is violated.
Thus, $F(z^{k+1}) - \widetilde{F}^k(z^{k+1}) > (1 - \beta)v_k$.
Using this in (\ref{eqn: gap}), we obtain
\[
\eta^{k+1} - \eta^k  \ge \frac{1 - \beta}{2}{\mu}_k v_k.
\]
Since $\mu_k \ge \bar{\mu}_k$, the
postulated bound \eqref{eta-increase} follows. 
\end{proof}


We recall a useful bound of the changes from $\eta^k$ to $\eta^{k+1}$ at descent steps.

\begin{lemma}
\label{l:eta-in-descent} If a descent step occurs at iteration $k$, then
\begin{equation}
\label{eta-descent}
\eta^{k+1} - \eta^{k}\ge - \rho \big\|x^{k+1}-x^k\big\|^2  \ge   \frac{1}{\beta} \big( F(x^{k+1}) - F(x^k)\big).
\end{equation}
\end{lemma}
\begin{proof}
See \cite[(7.68)-(7.69)]{ruszczynski2006nonlinear}.
\end{proof}

The following lemma relates the values of the optimal value of \eqref{subproblem}, $\eta^k$, and the value
$\widetilde{F}(z^{k+1})$ at the solution of \eqref{subproblem}.
\begin{lemma}
\label{l:eps-half}
At every iteration we have the inequality:
\[
F(x^k) - \eta^k \ge \frac{1}{2} \big[ F(x^k) - \widetilde{F}^k(z^{k+1})\big].
\]
\end{lemma}
\begin{proof}
Consider the function
\[
\varPhi(\tau) = (1-\tau)F(x^k) + \tau \widetilde{F}^k(z^{k+1})
+ \frac{\rho}{2} \big\| (1-\tau)x^k + \tau z^{k+1} - x^k\big\|^2.
\]
By construction, $\varPhi(1) = \eta^k$, and, due to the convexity of $\widetilde{F}^k(\cdot)$,
\begin{equation}
\label{Phibound}
\varPhi(\tau) \ge \widetilde{F}^k\big((1-\tau)x^k + \tau z^{k+1}\big)
+ \frac{\rho}{2} \big\| (1-\tau)x^k + \tau z^{k+1}- x^k\big\|^2, \quad \tau\in [0,1].
\end{equation}
By the definition of $z^{k+1}$, the right hand side of \eqref{Phibound} is minimized at $\tau=1$.
Therefore, $\varPhi'(1) \le 0$. Differentiating, we obtain the inequality
\[
- F(x^k) + \widetilde{F}^k(z^{k+1}) + \rho\big\| z^{k+1}-x^k\big\|^2 \le 0.
\]
This implies that
\begin{align*}
\eta^k &= \widetilde{F}^k(z^{k+1}) + \frac{\rho}{2}\big\| z^{k+1}-x^k\big\|^2 \le
\widetilde{F}^k(z^{k+1}) + \frac{1}{2}\big[F(x^k) - \widetilde{F}^k(z^{k+1})\big] \\
&= \frac{1}{2}\big[F(x^k) + \widetilde{F}^k(z^{k+1})\big].
\end{align*}
This is equivalent to the postulated inequality. 
\end{proof}

Finally, we recall the following bound of the Moreau--Yosida regularization.
\begin{lemma}
\label{l:lemma7.12} For any point $x\in \Rb^n$ we have
\begin{equation}
\label{lemma7.12}
F_\rho(x) \leq  F(x) - \big\|x - x^{*}\big\|^2\,\varphi\bigg( \frac{F(x) - F(x^{*})}{\big\|x - x^{*}\big\|^2}\bigg),
\end{equation}
where
\[
\varphi(t) =
\begin{cases}\; t^2 & \text{if\; $t\in[0,1]$},\\
\; -1 +2t & \text{if\; $t \ge 1$}.
\end{cases}
\]
\end{lemma}
\begin{proof} See \cite[Lem. 7.12]{ruszczynski2006nonlinear}. \end{proof}

\section{Rate of Convergence}
\label{s:rate}

Our objective in this section is to derive a worst-case bound on the rate of convergence of the method. To this end, we assume that $\varepsilon>0$
at Step 2 (inequality \eqref{eqn: n_v}) and we bound the number of iterations needed to achieve this accuracy.

We make a key assumption about strong convexity of the function $F(\cdot)$.
\begin{assumption}
\label{a:growth}
The function $F(\cdot)$ has a unique minimum point $x^*$ and a constant $\alpha >0$ exists, such that
\[
F(x) - F(x^{*}) \geq \alpha \big\|x - x^{*}\big\|^2,
\]
for all $x\in \Rb^n$ with $F(x) \le F(x^1)$.
\end{assumption}

We first show that stopping test of Step 2 guarantees the objective function accuracy of order $\varepsilon$.

\begin{lemma}\label{lem: conv_rate_1}
Suppose Assumption \ref{a:growth} is satisfied. Then at every iteration $k$ we have
\begin{equation}\label{eqn:stopping lemma_gen}
F(x^k) - F(x^{*}) \leq \frac{F(x^k)-\eta^k}{\min(\alpha,1)}.
\end{equation}
\end{lemma}
\begin{proof}
Since $\widetilde{F}^k(\cdot) \leq F(\cdot)$, we have
\begin{equation}\label{eqn:stoping inequality}
\begin{aligned}
F_\rho(x^k) &= \min_x\Big\{F(x) + \frac{\rho}{2}\big\|x - x^k\big\|^2 \Big\} \geq
\min_x \Big\{\widetilde{F}^k(x) + \frac{\rho}{2}\big\|x - x^k\big\|^2 \Big\} = \eta^k.
\end{aligned}
\end{equation}
Consider two cases.\\
\emph{Case 1:} If $F(x^k) - F(x^{*}) \le \big\|x^k - x^{*}\big\|^2$, then \eqref{lemma7.12} with $x=x^k$ yields
\[
F_\rho(x^k) \leq  F(x^k) - \frac{\big(F(x^k) - F(x^{*})\big)^2}{\big\|x^k - x^{*}\big\|^2}.
\]
Combining this inequality with \eqref{eqn:stoping inequality}, we conclude that
\[
 \frac{\big(F(x^k) - F(x^{*})\big)^2}{\big\|x^k - x^{*}\big\|^2} \leq F(x^k)-\eta^k.
\]
Substitution of the denominator by the upper bound $(F(x^k) - F(x^{*}))/\alpha$ implies \eqref{eqn:stopping lemma_gen}.\\
\emph{Case 2:}
$F(x^k) - F(x^{*})> \big\|x^k - x^{*}\big\|^2$.
Then \eqref{lemma7.12} yields
\[
F_\rho(x^k) \leq  F(x^k) - 2 \big( F(x^k) - F(x^{*}) \big) + \big\|x^k - x^{*}\big\|^2.
\]
With a view to \eqref{eqn:stoping inequality}, we obtain
\[
2 \big( F(x^k) - F(x^{*}) \big) - \big\|x^k - x^{*}\big\|^2 \le F(x^k)-\eta^k,
\]
which implies that
$F(x^k) - F(x^{*}) \le F(x^k)-\eta^k$
in this case. 
\end{proof}
\begin{corollary}
\label{c:stopping}
Suppose Assumption \ref{a:growth} is satisfied. If
the stoping test \eqref{eqn: n_v} is satisfied at iteration $k$, then
\begin{equation}\label{eqn:stopping lemma}
F(x^k) - F(x^{*}) \leq \frac{\varepsilon}{\min(\alpha,1)}.
\end{equation}
\end{corollary}

To bound the number of iterations of the method needed to achieve the prescribed
accuracy we consider two issues. First, we prove linear rate of convergence between descent steps. Then, we bound the numbers of null steps between consecutive descent steps.

By employing the bound of Lemma \ref{lem: conv_rate_1}, we can address the first issue.

\begin{lemma}\label{lem: conv_rate_2}
Suppose Assumption \ref{a:growth} is satisfied. Then at every descent step $k$ we have
\begin{equation}
\label{e:linear-rate}
F(z^{k+1}) - F(x^*) \leq (1 - \bar{\alpha}\beta)\big(F(x^k) - F(x^{*})\big),
\end{equation}
where $\bar{\alpha} = \min(\alpha,1)$.
\end{lemma}
\begin{proof}
It follows from the update rule (\ref{eqn: n_new_rule1}) that
\[
F(z^{k+1}) \leq 
 (1 - \beta)F(x^k) + \beta \widetilde{F}^k(z^{k+1}).
\]
Since $\widetilde{F}^k(z^{k+1}) \le \eta^k$, Lemma \ref{lem: conv_rate_1} yields
\[
F(x^k) - F(x^{*}) \leq \frac{1}{\bar{\alpha}} \big(F(x^k) - \widetilde{F}^k(z^{k+1})\big).
\]
Combining these inequalities and simplifying, we conclude that
\begin{align*}
F(z^{k+1}) &\leq (1 - \beta) F(x^k) + \beta \big( \bar{\alpha} F(x^{*}) - \bar{\alpha} F(x^k) + F(x^k)\big) \\
&=F(x^k) - \bar{\alpha}\beta\big( F(x^k) - F(x^{*})\big).
\end{align*}
Subtraction of $F(x^*)$ from both sides yields the linear rate \eqref{e:linear-rate}. 
\end{proof}

We now pass to the second issue of deriving an upper bound on the number of null steps between two consecutive
descent steps. To this end, we analyze the evolution of the gap $F(x^k)-\eta^k$.

It follows from \cite[(7.64)]{ruszczynski2006nonlinear}  that for all $k$
\[
\big\|x^k-x^*\big\|^2 \le \big\|x^1-x^*\big\|^2 + \frac{2(1-\beta)}{\beta\rho}\big[ F(x^1)-F(x^*)\big].
\]
Thus,  a uniform upper bound exists on the norm of the subgradients collected at points~$x^k$.
Therefore, a uniform upper bound exists on the distances $\|z^{k+1}-x^k\|$. Consequently,
the subgradients collected at the points $z^{k+1}$ are uniformly bounded as well, and the bound
depends on the starting point only. Consequently, a constant $M$ exists  such that
\[
\big\| s^{k+1}-g^{k+1}\big\|^2 \le \rho M
\]
at all null steps. With no loss of generality, we assume that $\varepsilon \le M$.

\begin{lemma}\label{lem: conv_rate_3}
If a null step occurs at iteration $k$, then
\begin{equation}
F(x^k) - \eta^{k+1} \le \gamma \big(F(x^k) - \eta^k\big),
\end{equation}
where
\begin{equation}
\label{tauk-bound}
\gamma = 1 - \frac{(1-\beta)^2\varepsilon}{2M}.
\end{equation}
\end{lemma}
\begin{proof}
By Lemma \ref{lem: increment_bound}, we have
\begin{equation}
F(x^k) - \eta^{k+1} \le F(x^k) - \eta^k - \frac{1-\beta}{2}\bar{\mu}_kv_k.
\end{equation}
On the other hand,
\begin{equation}
v_k = F(x^k) - \widetilde{F}^k(z^{k+1}) = F(x^k) - \eta^k + \frac{\rho}{2}\big\|z^{k+1} - x^k\big\|^2 \geq F(x^k) - \eta^k.
\end{equation}
Combining the last two inequalities, we conclude that
\begin{equation}
\label{gap-decrease}
\begin{split}
F(x^k) - \eta^{k+1} 
    &\le F(x^k) - \eta^k - \frac{1-\beta}{2}\bar{\mu}_k\big(F(x^k) - \eta^k\big)\\
    &= \biggl(1-\frac{1-\beta}{2}\bar{\mu}_k\biggr)\big(F(x^k) - \eta^k\big).
\end{split}
\end{equation}
Consider the definition \eqref{muk-def} of $\bar{\mu}_k$ in Lemma \ref{lem: increment_bound}.
If $\bar{\mu}_k = 1$, then $(1- \frac{1-\beta}{2}\bar{\mu}_k)$ is no greater than the bound \eqref{tauk-bound},
because $\varepsilon\le M$.
Otherwise, $\bar{\mu}_k$ is given by the second case in \eqref{muk-def}. Since the algorithm does not stop, we have $v_k > \varepsilon$, and thus
\[
\bar{\mu}_k = \frac{(1-\beta)\rho v_k}{\|s^{k+1} - g^{k+1}\|^2} \ge
\frac{(1-\beta) \varepsilon}{M}.
\]
Substitution to \eqref{gap-decrease} yields \eqref{tauk-bound}. 
\end{proof}

Let $x^{(\ell-1)}, x^{(\ell)}, x^{(\ell+1)}$ be three consecutive proximal centers for $\ell \geq 2$ in the algorithm. We want to bound the number of iterations made with proximal center $x^{(\ell)}$. To this end, we bound two quantities: $F(x^{(\ell)}) - \eta^{k(\ell)}$, where $k(\ell)$ is
the \emph{first} step with proximal center $x^{(\ell)}$, and
$F(x^{(\ell)}) - \eta^{k'(\ell)}$, where $k'(\ell)$ is the \emph{last} step with proximal center $x^{(\ell)}$.

%
%

In the following we discuss each issue separately.

Recall that according to the algorithm, $x^{(\ell)}$ is the optimal solution of the last subproblem with proximal center $x^{(\ell-1)}$. Let  $\eta^{k(\ell)-1}$ be the optimal objective value of the subproblem, that is,
\[
\eta^{k(\ell)-1} = \widetilde{F}^{k(\ell)-1}(x^{(\ell)}) + \frac{\rho}{2}\big\|x^{(\ell)} - x^{(\ell-1)}\big\|^2.
\]

\begin{lemma}\label{lem: null_step_start_bound} If a descent step is made at iteration $k(\ell)-1$, then
\begin{equation}
\label{eqn: null_step_upper_bound_1}
F(x^{(\ell)}) - \eta^{k(\ell)} \leq \frac{3}{2\beta}\big(F(x^{(\ell-1)}) - F(x^{(\ell)})\big).
\end{equation}
\end{lemma}
\begin{proof}
The left inequality in \eqref{eta-descent} yields
\[
\eta^{k(\ell)} \geq \eta^{k(\ell)-1} - \rho \big\|x^{(\ell)} - x^{(\ell-1)}\big\|^2.
\]
Since $F(x^{(\ell)}) \leq F(x^{(\ell-1)})$, we obtain
\[
F(x^{(\ell)}) - \eta^{k(\ell)} \leq F(x^{(\ell-1)}) - \eta^{k(\ell)-1} + \rho \big\|x^{(\ell)} - x^{(\ell-1)}\big\|^2.
\]
As iteration $k(\ell)-1$ is a descent step, the update rule (\ref{eqn: n_new_rule1}) holds. Thus
\begin{align*}
F(x^{(\ell-1)}) - \eta^{k(\ell)-1} &= \biggl[F(x^{(\ell-1)}) - \widetilde{F}^{k(\ell)-1}(x^{(\ell)})\biggr] - \frac{\rho}{2}\big\|x^{(\ell)} - x^{(\ell-1)}\big\|^2\\
 &\leq \frac{1}{\beta}\big(F(x^{(\ell-1)}) - F(x^{(\ell)})\big) - \frac{\rho}{2}\big\|x^{(\ell)} - x^{(\ell-1)}\big\|^2.
\end{align*}
Combining the last two inequalities we obtain
\[
F(x^{(\ell)}) - \eta^{k(\ell)} \leq \frac{1}{\beta}\big(F(x^{(\ell-1)}) - F(x^{(\ell)})\big) + \frac{\rho}{2}\big\|x^{(\ell)} - x^{(\ell-1)}\big\|^2.
\]
The right inequality in \eqref{eta-descent} can be now used to substitute $\big\|x^{(\ell)} - x^{(\ell-1)}\big\|^2$ on the right hand side
to obtain \eqref{eqn: null_step_upper_bound_1}. 
\end{proof}

We can now integrate our results.

Applying Lemma \ref{lem: conv_rate_1}, we obtain the following inequality at \emph{every} null step with prox center~$x^{(\ell)}$:
\begin{equation}
\label{e:limitbd}
\begin{aligned}
F(x^{(\ell)}) - \eta^k &\ge \bar{\alpha}\big( F(x^{(\ell)}) - F(x^*)\big)
                     \ge \bar{\alpha}\big( F(x^{(\ell)}) - F(x^{\ell +1})\big).
\end{aligned}
\end{equation}
From Lemma \ref{lem: null_step_start_bound} we know that for $2 \le \ell < L$, where $L$ is the last proximal center,
 the initial value of the left hand side (immediately after the
previous descent step) is bounded from above by the expression on the right hand side
of \eqref{eqn: null_step_upper_bound_1}.
Lemma \ref{lem: conv_rate_3} established a linear rate of decrease of the left hand side of
\eqref{e:limitbd}. Therefore, the number $n_\ell$ of null steps
with proximal center $x^{(\ell)}$, if it is positive, satisfies the inequality:
\[
\frac{3}{2\beta}\big(F(x^{(\ell-1)}) - F(x^{(\ell)})\big) \gamma^{n_\ell - 1} \ge \bar{\alpha}\big( F(x^{(\ell)}) - F(x^{(\ell +1)})\big).
\]
Consequently, for $2 \le \ell <L$ we obtain the following upper bound on the number of null steps:
\begin{equation}
\label{nl}
n_\ell \leq 1 + \frac{1}{\ln(\gamma)}  \ln\left(\frac{2\beta\bar{\alpha}}{3}\frac{F(x^{(\ell)}) - F(x^{(\ell+1)})}{F(x^{(\ell-1)}) - F(x^{(\ell)})}\right).
\end{equation}
If the number $n_\ell$ of null steps is zero, inequality \eqref{e:linear-rate} yields
\begin{align*}
\frac{F(x^{(\ell)}) - F(x^{(\ell+1)})}{F(x^{(\ell-1)}) - F(x^{(\ell)})}
&\le \frac{F(x^{(\ell)}) - F(x^{*})}{F(x^{(\ell-1)})- F(x^*) - \big(F(x^{(\ell)})- F(x^*)\big)}
 \le
\frac{1}{\frac{1}{1-\bar{\alpha}\beta}-1}.
\end{align*}
Elementary calculations then prove that both logarithms on the right hand side of \eqref{nl} are negative, and thus inequality \eqref{nl} is satisfied in this case as well.

Suppose there are $L$ proximal centers appearing throughout the algorithm: $x^{(1)}$, $x^{(2)}$, \dots, $x^{(L)}$. They divide the progress of the algorithm into $L$ series of null steps. For the first series, similar to the analysis above, we use \eqref{e:limitbd} and Lemma \ref{lem: conv_rate_3} to obtain the bound
\[
n_1 \leq 1 + \frac{1}{\ln(\gamma)}  \ln\left(\bar{\alpha}
\frac{F(x^{(1)}) - F(x^{(2)})}{F(x^{(1)}) - \eta^1}\right).
\]
For the last series, we use Lemma \ref{l:eps-half} to derive the inequality $F(x^{(\ell)}) - \eta^k \ge \varepsilon/2$, which must hold at every iteration at which the stopping test
is not satisfied. We use it instead of \eqref{e:limitbd}  in our analysis, and we obtain
\[
n_L \leq 1 + \frac{1}{\ln(\gamma)}  \ln\left(\frac{\beta}{3}\frac{\varepsilon}{F(x^{(L-1)}) - F(x^{(L)})}\right).
\]
We aggregate the total number of null steps for different proximal centers
and we obtain the following bound:
\begin{equation}\label{null step bound}
\begin{split}
\sum_{\ell=1}^L n_\ell &\le \frac{L-1}{\ln(\gamma)}\left[ \ln(\bar{\alpha})+\ln\left(\frac{2\beta\bar{\alpha}}{3}\right)
+ \ln\left(\frac{\beta}{3}\right)
+ \frac{1}{L-1}\ln\left( \frac{\varepsilon}{F(x^{1}) - \eta^1}\right)\right]
 + L.
\end{split}
\end{equation}
Let us recall the definition of $\gamma$ in \eqref{tauk-bound}, and denote
\[
C = \frac{(1-\beta)^2}{2M},
\]
so that $\gamma = 1 - \varepsilon {C}$. Since
$\ln(1-\varepsilon C)< -\varepsilon C$, we derive the following inequality for the number of null steps:
\begin{equation}\label{null step bound3}
\sum_{\ell=1}^L n_\ell \le \frac{L-1}{-\varepsilon C}\left[\ln(\bar{\alpha})+\ln\left(\frac{2\beta\bar{\alpha}}{3}\right)
+ \ln\left(\frac{\beta}{3}\right)
+ \frac{1}{L-1}\ln\left( \frac{\varepsilon}{F(x^{1}) - \eta^1}\right)\right] + L.
\end{equation}
Let us now derive an upper bound on the number $L$ of descent steps. By virtue of \eqref{eqn: n_v}
and \eqref{eqn: n_new_rule1}, descent steps are made only if
\[
F(x^k) - F(x^*) \ge \beta\varepsilon;
\]
otherwise, the method must stop. To explain it more specifically, if $F(x^k) - F(x^*) \le \beta\varepsilon$, then $F(x^k) - F(z^{k+1}) \le \beta\varepsilon$. If a descent step is made, $F(z^{k+1}) \le F(x^k) - \beta v_k$. Then $\beta v_k \le \beta\varepsilon $, $v_k \le \varepsilon$. Thus we can't make a descent step because the algorithm has already stopped, which contradicts our assumption.
 It follows from Lemma \ref{lem: conv_rate_2}, that
\[
(1 - \bar{\alpha} \beta)^{L-1} \big(F(x^1) - F(x^{*})\big) \ge \beta\varepsilon.
\]
Therefore,
\begin{equation}
\label{decent step bound}
L \leq 1 + \frac{\ln(\beta\varepsilon) - \ln\big(F(x^1) - F(x^{*})\big)}{\ln( 1 - \bar{\alpha} \beta)} .
\end{equation}
 As a result, we have the final bound for the total number of descent and null steps:
\begin{equation}\label{total step bound}
\begin{split}
\lefteqn{{L+\sum_{\ell=1}^L n_\ell}}\\
 &\le
\frac{1}{\varepsilon C\ln( 1 - \bar{\alpha} \beta)}\ln\left(\frac{F(x^1) - F(x^{*})}{\beta\varepsilon}\right)\Bigg[\ln(\bar{\alpha})+\ln\left(\frac{2\beta\bar{\alpha}}{3}\right)+ \ln\left(\frac{\beta}{3}\right)\Bigg] \\
&{\quad} + \frac{1}{\varepsilon C}\ln\left( \frac{F(x^{1}) - \eta^1}{\varepsilon}\right)
+ 2\frac{\ln(\beta\varepsilon) - \ln\big(F(x^1) - F(x^{*})\big)}{\ln( 1 - \bar{\alpha} \beta)} + 2.
\end{split}
\end{equation}

Therefore in order to achieve precision $\varepsilon$, the number of steps needed is of order
\[
L+\sum_{\ell=1}^L n_\ell \sim \mathcal{O}\Bigg( \frac{1}{\varepsilon}\ln\bigg(\frac{1}{\varepsilon}\bigg)\Bigg).
 \]
This is almost equivalent to saying that given the number of iterations $k$, the precision of the solution is approximately $\mathcal{O}(1/k)$.


\begin{thebibliography}{10}

\bibitem{lem:iiasa}
C.~Lemar{\'e}chal.
\newblock Nonsmooth optimization and descent methods.
\newblock Research {R}eport 78-4, International Institute of Applied Systems
  Analysis, Laxenburg, Austria, 1978.

\bibitem{mif:82}
R.~Mifflin.
\newblock A modification and an extension of {L}emar{\'e}chal's algorithm for
  nonsmooth minimization.
\newblock In D.~C. Sorensen and R.~J.~B. Wets, editors, {\em Nondifferential
  and Variational Techniques in Optimization}, volume~17, pages 77--90. 1982.

\bibitem{kiwiel1983aggregate}
K.~C. Kiwiel.
\newblock An aggregate subgradient method for nonsmooth convex minimization.
\newblock {\em Mathematical Programming}, 27(3):320--341, 1983.

\bibitem{Kiwiel-book}
K.~C. Kiwiel.
\newblock {\em Methods of Descent for Nondifferentiable Optimization}, volume
  1133 of {\em Lecture Notes in Mathematics}.
\newblock Springer-Verlag, Berlin, 1985.

\bibitem{HUL-book}
J.-B. Hiriart-Urruty and C.~Lemar{\'e}chal.
\newblock {\em Convex Analysis and Minimization Algorithms. {II}}, volume 306
  of {\em Grundlehren der Mathematischen Wissenschaften [Fundamental Principles
  of Mathematical Sciences]}.
\newblock Springer-Verlag, Berlin, 1993.

\bibitem{bogilesa:on}
J.-F. Bonnans, J.~C. Gilbert, C.~Lemar{\'e}chal, and C.~Sagastiz{\'a}bal.
\newblock {\em Numerical Optimization. Theoretical and Practical Aspects}.
\newblock Springer-Verlag, Berlin, 2003.

\bibitem{LNN95}
C.~Lemar\'{e}chal, A.~Nemirovskii, and Y.~Nesterov.
\newblock New variants of bundle methods.
\newblock {\em Mathematical Programming}, 69(1-3):111--147, 1995.

\bibitem{KKC95}
K.~C. Kiwiel.
\newblock Proximal level bundle methods for convex nondifferentiable
  optimization, saddle-point problems and variational inequalities.
\newblock {\em Mathematical Programming}, 69(1-3):89--109, 1995.

\bibitem{LG15}
G.~Lan.
\newblock Bundle-level type methods uniformly optimal for smooth and nonsmooth
  convex optimization.
\newblock {\em Mathematical Programming}, 149(1-2):1--45, 2015.

\bibitem{rockafellar1976monotone}
R.~T. Rockafellar.
\newblock Monotone operators and the proximal point algorithm.
\newblock {\em SIAM Journal on Control and Optimization}, 14(5):877--898, 1976.

\bibitem{ruszczynski2006nonlinear}
A.~Ruszczy{\'n}ski.
\newblock {\em Nonlinear Optimization}.
\newblock Princeton University Press, 2006.

\end{thebibliography}
\end{document}